\documentclass[11pt, reqno]{article}
\usepackage{amsmath,amssymb,amsfonts,amsthm}
\usepackage{color}

\usepackage[colorlinks=true,linkcolor=blue,citecolor=blue,urlcolor=blue,pdfborder={0 0 0}]{hyperref}
\usepackage{cleveref}

\usepackage{caption}
\usepackage{thmtools}

\usepackage{mathrsfs}

\crefname{theorem}{Theorem}{Theorems}
\crefname{thm}{Theorem}{Theorems}
\crefname{lemma}{Lemma}{Lemmas}
\crefname{lem}{Lemma}{Lemmas}
\crefname{remark}{Remark}{Remarks}
\crefname{prop}{Proposition}{Propositions}
\crefname{defn}{Definition}{Definitions}
\crefname{corollary}{Corollary}{Corollaries}
\crefname{conjecture}{Conjecture}{Conjectures}
\crefname{question}{Question}{Questions}
\crefname{chapter}{Chapter}{Chapters}
\crefname{section}{Section}{Sections}
\crefname{figure}{Figure}{Figures}

\theoremstyle{plain}
\newtheorem{thm}{Theorem}[section]

\newtheorem{lemma}[thm]{Lemma}
\newtheorem{theorem}[thm]{Theorem}

\newtheorem{corollary}[thm]{Corollary}

\newtheorem{prop}[thm]{Proposition}

\newtheorem{question}[thm]{Question}
\theoremstyle{definition}

\theoremstyle{remark}
\newtheorem{remark}[thm]{Remark}

\numberwithin{equation}{section}

\renewcommand{\P}{\mathbb P}
\newcommand{\E}{\mathbb E}

\newcommand{\Z}{\mathbb Z}

\newcommand{\sA}{\mathscr A}

\newcommand{\sH}{\mathscr H}

\usepackage[margin=0.9in]{geometry}
\usepackage{extarrows}
\usepackage{titling}
\usepackage{commath}
\usepackage{mathtools}
\usepackage{titlesec}
\newcommand{\eps}{\varepsilon}
\usepackage{bbm}
\usepackage{setspace}
\usepackage{enumitem}

\usepackage[textsize=tiny]{todonotes}

\usepackage{cite}

\newcommand{\bP}{\mathbf P}

\newcommand{\lrDini}{\left(\frac{d}{dp}\right)_{\hspace{-0.2em}+}\!}

\def\P{\mathbb{P}}

\DeclareMathSymbol{\leqslant}{\mathalpha}{AMSa}{"36} 
\DeclareMathSymbol{\geqslant}{\mathalpha}{AMSa}{"3E} 
\DeclareMathSymbol{\eset}{\mathalpha}{AMSb}{"3F}     

\renewcommand{\epsilon}{\varepsilon}

\pretitle{\begin{flushleft}\Large}
\posttitle{\par\end{flushleft}\vskip 0.5em
}
\preauthor{\begin{flushleft}}
\postauthor{
\\
\vspace{0.5em}
\footnotesize{The Division of Physics, Mathematics and Astronomy, California Institute of Technology\\ Email: 
\href{mailto:t.hutchcroft@caltech.edu}{t.hutchcroft@caltech.edu}}
\end{flushleft}}
\predate{\begin{flushleft}}
\postdate{\par\end{flushleft}}
\title{\bf Transience and anchored isoperimetric dimension of supercritical percolation clusters
}

\renewenvironment{abstract}
 {\par\noindent\textbf{\abstractname.}\ \ignorespaces}
 {\par\medskip}

\titleformat{\section}
  {\normalfont\large \bf}{\thesection}{0.4em}{}

\author{{\bf Tom Hutchcroft}}
\begin{document}

\date{\small{\today}}

\maketitle

\setstretch{1.1}

\begin{abstract}
 We establish several equivalent characterisations of the anchored isoperimetric dimension of supercritical clusters in  Bernoulli bond percolation on transitive graphs.
 We deduce from these characterisations together with a theorem of Duminil-Copin, Goswami, Raoufi, Severo, and Yadin (\emph{Duke Math.\ J.\ 2020}) that if $G$ is a transient transitive graph then the infinite clusters of Bernoulli percolation on $G$ are transient for $p$ sufficiently close to $1$. It remains open to extend this result down to the critical probability. Along the way we establish two new \emph{cluster repulsion inequalities} that are of independent interest.
\end{abstract}






\section{Introduction}\label{sec:intro}

Let $G=(V,E)$ be a connected, locally finite graph. In \textbf{Bernoulli bond percolation}, the edges of $G$ are each either deleted or retained independently at random, with retention probability $p\in [0,1]$, to obtain a random subgraph $G_p$ of $G$. Retained edges are referred to as \textbf{open} and deleted edges are referred to as \textbf{closed}, with the connected components of the open subgraph $G_p$ referred to as \textbf{clusters}. We will be primarily interested in the case that $G$ is \textbf{transitive}, meaning that for any two vertices $u$ and $v$ of $G$ there exists an automorphism mapping $u$ to $v$.  When $G$ is infinite, the \textbf{critical probability} $p_c=p_c(G)$ is defined by
\[
p_c=\inf\bigl\{p\in [0,1]: G_p \text{ has an infinite cluster almost surely}\bigr\},
\]
which typically satisfies $p_c<1$ once obvious `one-dimensional' counterexamples such as the line graph $\Z$ are excluded \cite{MR4181032,MR3520023,MR3865659,hutchcroft2021non}. 

Many of the most interesting questions concerning percolation in the supercritical phase $p_c<p <1$ can be phrased as follows: To what extent do the geometries of the infinite clusters of $G_p$ resemble the geometry of the ambient graph $G$? In particular, does the random walk on the infinite clusters of $G_p$ have similar behaviour to the random walk on $G$?
These questions are already interesting in the perturbative regime $1-p \ll 1$ where they are closely related to models in which the edges of $G$ are assigned independent random lengths \cite{MR1395617}. In the classical case of the hypercubic lattice $\Z^d$, a rich and detailed theory has been developed following the foundational works of Grimmett and Marstrand \cite{MR1068308} and Antal and Pisztora \cite{MR1404543}, with further significant works establishing the transience of the infinite cluster for $d\geq 3$ \cite{MR1222363}, the sharp isoperimetric inequalities and heat kernel estimates \cite{mathieu2004isoperimetry,MR2094438,Pete08,MR4417202}, and convergence of the random walk on the infinite cluster to Brownian motion \cite{MR2278453,MR2345229}. See also \cite{MR4417202,MR4288333,MR3077517,MR3144026,MR3000561,MR2433935,MR1896880} for related results on long-range models. A systematic study of similar questions beyond $\Z^d$ was initiated in the seminal work of Benjamini, Lyons, and Schramm \cite{BLS99}.
 Despite significant progress of many authors \cite{HermonHutchcroftSupercritical,contreras2021supercritical,benjamini2015disorder,MR2221157,chen2004anchored,v00}, a comprehensive theory of supercritical percolation on transitive graphs is still to be developed and several very basic problems remain open.

In this paper we are interested primarily in the \emph{isoperimetry} of the infinite clusters, that is, the boundary/volume ratios of finite subsets of the cluster. Understanding the isoperimetry of a graph is closely entwined with understanding the behaviour of random walk on that graph, see e.g.\ \cite{KumagaiBook,LP:book} for background.
Let $G$ be a connected, locally finite graph and let $d>1$. We say that $G$ satisfies a \textbf{uniform\footnote{An inequality of this form is usually referred to simply as an isoperimetric inequality. We prepend the word `uniform' to distinguish from the other kinds of isoperimetric inequalities we consider.} $d$-dimensional isoperimetric inequality} if there exists a positive constant $c$ such that
\[
|\partial_E W| \geq c\cdot \Bigl(\sum_{w\in W} \deg(w)\Bigr)^{(d-1)/d}
\]
for every finite set of vertices $W \subseteq V$, where $\partial_E W$ denotes the set of edges with one endpoint in $W$ and the other in $V\setminus W$. The \textbf{uniform isoperimetric dimension} of $G$ is defined to be the supremal value of $d$ for which $G$ satisfies a uniform $d$-dimensional isoperimetric inequality.
It is a consequence of a theorem of Coulhon and Saloff-Coste \cite{MR1232845} that a transitive graph satisfies a uniform $d$-dimensional isoperimetric inequality if and only if it has at least $d$-dimensional volume growth in the sense that its balls $B(v,r)$ satisfy $|B(v,r)|\geq cr^d$ for some $c>0$. Moreover, the classification of transitive graphs of polynomial volume growth due to Gromov \cite{Gromov87} and Trofimov \cite{MR811571} implies that every transitive graph satisfying $|B(v,r)| \leq C r^C$ for some $C<\infty$ and infinitely many $r$ must satisfy $|B(v,r)| \asymp r^d$ for some \emph{integer} $d \geq 0$, where we write $\asymp$ for an equality holding to within two positive constants. Thus, every transitive graph of polynomial growth has a well-defined integer dimension $d$ that describes both its volume growth and uniform isoperimetric dimension, while transitive graphs of superpolynomial growth have infinite uniform isoperimetric dimension.

Since for $p_c<p<1$ the infinite clusters of $G_p$ always contains arbitrarily large `bad zones' whose induced subgraphs are isomorphic to paths, they cannot satisfy any non-trivial uniform isoperimetric inequality. These considerations led Benjamini, Lyons, and Schramm \cite{BLS99} to introduce the notion of \emph{anchored isoperimetric inequalities}, which are similar in spirit to uniform isoperimetric inequalities but weak enough to potentially hold in non-trivial random examples.
 We say that a graph $G=(V,E)$ satisfies an \textbf{anchored $d$-dimensional isoperimetric inequality} if for some (and hence every) vertex $v$ of $G$ there exists a positive constant $c(v)$ such that
\[
|\partial_E W| \geq c(v)\cdot \Bigl(\sum_{w\in W} \deg(w)\Bigr)^{(d-1)/d}
\]
for every finite connected set of vertices $W \subseteq V$ that contains $v$. 
 As before, the \textbf{anchored isoperimetric dimension} of $G$ is defined to be the supremal value of $d$ for which $G$ satisfies an anchored $d$-dimensional isoperimetric inequality. It follows from a theorem of Thomassen \cite{Thomassen92} that graphs of anchored isoperimetric dimension strictly larger than $2$ are transient for simple random walk.

 Our main theorem establishes four equivalent characterisations of the anchored isoperimetric dimension of supercritical percolation clusters on a transitive graph. Before stating this theorem let us briefly introduce some relevant notation. Given a set $S \subseteq V$ and a vertex $v\notin S$, we write $\{S \leftrightarrow \infty\}$ for the event that $S$ is connected to infinity by an open path, write $\{S \nleftrightarrow \infty\}$ for the complement of this event, and write $\{ v \leftrightarrow \infty$ off $S\}$ for the event that $v$ is connected to infinity by an open path that does not visit any vertex of $S$. We write $\partial_E^\rightarrow S$ for the set of oriented edges $e$ with $e^-\in S$ and $e^+ \notin S$; although our graphs are unoriented it is convenient to think of each edge as corresponding to a pair of oriented edges. We also fix an arbitrary `origin' vertex $o$ of $V$ and write $K$ for the cluster of this vertex in $G_p$.

\begin{theorem}[Characterisation of the isoperimetric dimension of percolation clusters]
\label{thm:main_equivalence}
Let $G=(V,E)$ be a connected, locally finite, transitive graph, let $p_c\leq p_0< 1$ and let $d \in (1,\infty]$. The following are equivalent:
\begin{enumerate}[label=\emph{(\roman*)}]
\item For each $p_0<p\leq 1$ and $d'<d$ there exists a positive constant $c=c(p,d')$ such that
\begin{equation}
\label{eq:Sbound}
\P_p (S\nleftrightarrow \infty) \leq \exp\left[ -c|S|^{(d'-1)/d'}\right]
\end{equation}
for every finite, non-empty set $S\subseteq V$.
\item For each $p_0<p\leq 1$ and $d'<d$ there exists a positive constant $c=c(p,d')$ such that
\begin{equation}
\label{eq:Sboundary}
\sum_{e \in \partial_E^\rightarrow S} \P_p\left(e^+ \leftrightarrow \infty \text{ \emph{off} $S$}\right) \geq c|S|^{(d'-1)/d'}
\end{equation}
for every finite, non-empty set $S\subseteq V$.
\item For each $p_0<p\leq 1$ and $d'<d$ there exists a positive constant $c=c(p,d')$ such that
\begin{equation}
\label{eq:volume_tail}
\P_p(n\leq |K|<\infty) \leq \exp\left[ -cn^{(d'-1)/d'}\right]
\end{equation}
for every $v\in V$ and $n\geq 1$.
\item For each $p_0<p\leq 1$, every infinite cluster of $G_p$ has anchored isoperimetric dimension at least $d$ almost surely.
\end{enumerate}
The equivalence \emph{(i)} $\Leftrightarrow$ \emph{(ii)} and the implications \emph{(ii)}  $\Rightarrow$  \emph{(iii)} $\Rightarrow$ \emph{(iv)} do not require transitivity and hold for any bounded degree graph.
\end{theorem}

\begin{remark}
The freedom to reduce the dimension by an arbitrarily small amount is needed in the proof of (iii) $\Rightarrow$ (iv) but not in any other step of the proof. Quantitatively, the proof of (iii) $\Rightarrow$ (iv) yields that if the bound \eqref{eq:volume_tail} holds for some specific $p_c<p<1$ and $d'>1$ then every infinite cluster of $G_p$ satisfies an anchored $(\log t)^{-1} t^{(d'-1)/d'}$-isoperimetric inequality almost surely, see \cref{prop:union_bound_isoperimetry}. 
\end{remark}

\begin{remark}
The proofs of the individual propositions making up the various implications of \cref{thm:main_equivalence} all allow one to consider isoperimetric inequalities defined in terms of a general function $\phi$, which we take to be $\phi(t)=t^{(d-1)/d}$ when applying these propositions to prove \cref{thm:main_equivalence}. 
\end{remark}

In the classical case of $\Z^d$, the tail bound \eqref{eq:volume_tail} was proven to hold with $d'=d$ for every $p_c<p<1$ by Kesten and Zhang~\cite{MR1055419} conditional on the then-conjectural Grimmett-Marstrand theorem \cite{MR1068308} (see \cite{MR594824} for a matching lower bound), and the infinite cluster was proven to satisfy an anchored $d$-dimensional isoperimetric inequality by Pete \cite{Pete08} (see also \cite{mathieu2004isoperimetry}); \cref{thm:main_equivalence} shows that these two results are equivalent up to changing the relevant exponents by an arbitrarily small amount. Beyond the Euclidean setting, Contreras, Martineau, and Tassion \cite{contreras2021supercritical} have recently shown that if $G$ is a $d$-dimensional transitive graph then \eqref{eq:volume_tail} holds with $d'=d$ for every $p_c<p<1$. Thus, \cref{thm:main_equivalence} allows us to conclude in this setting that the infinite clusters have anchored isoperimetric dimension $d$ almost surely for every $p_c<p<1$. This result is new outside the Euclidean case.

\begin{corollary}
\label{cor:polygrowth}
Let $G=(V,E)$ be a connected, locally finite, transitive graph. If $G$ has polynomial volume growth of dimension $d$ then the infinite cluster of $G_p$ has anchored isoperimetric dimension $d$ almost surely for every $p_c<p\leq 1$.
\end{corollary}

As mentioned above, supercritical percolation remains very poorly understood on general transitive graphs: while we expect that infinite supercritical percolation clusters on a transitive graph of superpolynomial growth should have infinite isoperimetric dimension almost surely, our current understanding is not even sufficient to prove that this dimension is strictly greater than $1$! 

The one case where the problem is well-understood is the \emph{nonamenable} case, in which $G$ satisfies an isoperimetric inequality of the form $|\partial_E W| \geq c|W|$, where we established in joint work with Hermon \cite{HermonHutchcroftSupercritical} that for each $p_c<p<1$ there exists $c_p>0$ such that $\P_p(n \leq |K| < \infty) \leq e^{-c_p n}$ and hence that
infinite clusters have \emph{anchored expansion} almost surely for every $p_c<p<1$. A key step in this proof \cite[Proposition 2.4]{HermonHutchcroftSupercritical} was to establish an inequality of the form
\[
\sum_{e \in \partial_E^\rightarrow S} \P_p\left(e^+ \leftrightarrow \infty \text{ off $S$}\right) \geq c|S|,
\]
the analogue of \eqref{eq:Sboundary} in the nonamenable case.
The methods of this paper yield a much simpler method for concluding the rest of the proof given this estimate, see \cref{sec:repulsion}.

Although we currently lack tools to analyze the whole supercritical regime $p_c<p<1$ on amenable transitive graphs of superpolynomial growth, we are able to apply \cref{thm:main_equivalence} in conjunction with the results of Duminil-Copin, Goswami, Raoufi, Severo, and Yadin \cite{MR4181032} to prove that the infinite clusters have infinite isoperimetric dimension for $p$ sufficiently close to 1. Before stating their result, we recall that the \textbf{capacity} $\operatorname{Cap}(S)$ of a finite set of vertices in a graph $G$ is defined to be $\operatorname{Cap}(S)=\sum_{v\in S} \deg(v)\bP_v(\tau^+_S = \infty)$, where $\bP_v(\tau^+_S = \infty)$ denotes the probability that a simple random walk started at $v$ never returns to $S$ after time zero. It is a consequence of a theorem of Lyons, Morris, and Schramm \cite[Theorem 6.1]{LMS08} (see also \cite[Theorem 6.41]{LP:book}) that if $G$ satisfies a uniform $d$-dimensional isoperimetric inequality for some $d>2$ then there exists a positive constant $c$ such that $\operatorname{Cap}(S)\geq c|S|^{(d-2)/d}$ for every finite set of vertices $S$.

\begin{thm}[DGRSY] Let $G=(V,E)$ be a bounded degree graph satisfying a uniform $d$-dimensional isoperimetric inequality for some $d>4$. Then there exist $p_0<1$ and $c>0$ such that
\begin{equation}
\label{eq:DGRSY}
\P_p(S \nleftrightarrow \infty) \leq \exp\left[ -\frac{1}{2} \operatorname{Cap}(S)\right] \leq \exp\left[ -c |S|^{(d-2)/d}\right] 
\end{equation}
for every $p_0\leq p \leq 1$ and every finite set $S \subseteq V$.
\end{thm}

Note that for Cayley graphs of \emph{finitely presented} groups, a significantly easier argument is available due to Babson and Benjamini \cite{MR1622785} that establishes sharp upper bounds on $\P_p(S \nleftrightarrow \infty)$ for $p$ close to $1$ via a Peierls-type argument.



Since the implications (i) $\Rightarrow$ (ii) $\Rightarrow$ (iii) $\Rightarrow$ (iv) of \cref{thm:main_equivalence} do not require transitivity, and since the exponent $(d-2)/d=1-2/d$ appearing in \eqref{eq:DGRSY} can also be written as $((d/2)-1)/(d/2)$, we obtain the following corollary. 

\begin{corollary}
Let $G=(V,E)$ be a connected, bounded degree graph with uniform isoperimetric dimension $d \in (4,\infty]$. Then there exists $p_0=p_0(G)<1$ such that if $p>p_0$ then every infinite cluster of $G_p$ is transient and has anchored isoperimetric dimension at least $d/2$ almost surely. 
\end{corollary}

In conjunction with \cref{cor:polygrowth}, we deduce that if $G$ is a transient transitive graph then there exists $p_0<1$ such that if $p_0<p<1$ then every infinite cluster of $G_p$ is transient almost surely; this result was conjectured to hold for all $p>p_c$ by Benjamini, Lyons, and Schramm \cite[Conjecture 1.7]{BLS99}. As observed in \cite{MR1395617}, this is equivalent to the statement that $G$ remains transient a.s.\ when its edges are given i.i.d.\ finite resistances regardless of what the law of these resistances are. Previously, the same conclusion was established for graphs admitting an \emph{exponential intersection tail} (EIT) measure on paths in the work of Benjamini, Peres, and Pemantle \cite{MR1634419} (see also \cite{benjamini2022oriented}); it is currently unclear whether every transitive transient graph admits such a measure. (On the other hand, the EIT approach has recently been shown to be useful in the study of symmetry breaking in continuous-symmetry spin systems \cite{garban2021continuous,abbe2018group} in three and more dimensions, for which our methods do not seem relevant.)

\medskip

\noindent
\textbf{About the proof.} 
The equivalence (i) $\Leftrightarrow$ (ii) is a consequence of standard facts about increasing events. The implication (ii) $\Rightarrow$ (iii) is a consequence of the \emph{cluster repulsion inequalities} we develop in \cref{sec:repulsion}. Indeed, perhaps the most important insight of the paper is that it is significantly easier to prove \emph{sprinkled}  cluster repulsion inequalities than it is to prove the more usual forms of these inequalities, allowing us to sidestep the more difficult steps of \cite{HermonHutchcroftSupercritical,Pete08}. Although it is not used in the proof of the main theorem, we also show how a related cluster repulsion inequality \emph{without} sprinkling can be proven using martingale techniques inspired by the classical work of Aizenman, Kesten, and Newman \cite{MR901151}. Finally, we show how the methods of Pete \cite{chen2004anchored} yield the implication (iii) $\Rightarrow$ (iv) in \cref{subsec:isoperimetry} and complete the remaining steps of the proof in \cref{subsec:mainproof}.

\section{Cluster repulsion inequalities}
\label{sec:repulsion}

A central contribution of this paper is the establishment of new \emph{cluster repulsion inequalities} stating that it is hard for a finite cluster to touch the infinite cluster in a large number of places. The use of such inequalities was pioneered by Pete \cite{Pete08}, who proved a very strong cluster repulsion inequality for percolation on $\Z^d$ and deduced that the infinite supercritical cluster always satisfies an anchored $d$-dimensional isoperimetric inequality. Letting $\tau(A,B)$ denote the number of edges with one endpoint in $A$ and the other in $B$ and letting $K_\infty$ denote the union of infinite clusters in $G_p$, Pete proved that if $G=\Z^d$ and $p_c<p<1$ then there exists a positive constant $c=c(p,d)$ such that
\[
\P_p(m \leq |K| < \infty \text{ and } \tau(K,K_\infty) \geq t) \leq \exp\left[-c \max\{m^{(d-1)/d},t\}\right]
\]
for every $m \geq 1$ and $t\geq 0$. Once this inequality is established, one can deduce the anchored $d$-dimensional isoperimetric inequality for the infinite cluster by a simple counting argument similar to that carried out in the proof of \cref{prop:union_bound_isoperimetry} below. Pete's proof uses a renormalization argument and is not applicable to transitive graphs beyond the Euclidean setting.

The main proposition of this section establishes a simple \emph{sprinkled} cluster repulsion inequality that holds for any graph. Let $0<p_1<p_2<1$ and consider the standard monotone coupling of $G_{p_1}$ and $G_{p_2}$ as described in \cite[Chapter 2]{grimmett2010percolation}, where conditional on $G_{p_1}$ every open edge of $G_{p_1}$ is open in $G_{p_2}$ and the closed edges of $G_{p_1}$ are open or closed in $G_{p_2}$ independently at random with probability $(p_2-p_1)/(1-p_1)$ to be open. We write $K_{v,p_2}$ for the cluster of $v$ in $G_{p_2}$ and write $K_{\infty,p_1}$ for the union of all infinite clusters in $G_{p_1}$.

\begin{prop}[Cluster repulsion with sprinkling]
\label{prop:cluster_repulsionI} Let $G=(V,E)$ be a connected, locally finite graph, and let $0<p_1 <p_2 <1$. The inequality
\[
\P\left(|K_{v,p_2}|<\infty \text{ and } \tau\bigl(K_{v,p_2},K_{\infty,p_1}\bigr) \geq n\right) \leq \left(\frac{1-p_2}{1-p_1}\right)^n
\]
holds for every $v\in V$ and $n\geq 1$.
\end{prop}

The proof of \cref{prop:cluster_repulsionI} is not hard, and it is surprising that it has not been observed before. As mentioned above, it allows us to significantly simplify the analysis of \cite{HermonHutchcroftSupercritical}. It also easily leads us to the implication (ii) $\Rightarrow$ (iii) of \cref{thm:main_equivalence} via the following more general corollary. Given a set of vertices $S$ in a graph $G$ we write $E(S)$ for the set of edges that \textbf{touch} $S$, i.e., have at least one endpoint in $S$. By a slight abuse of notation, we will also write $E(K_v)$ for the set of edges that touch $K_v$ (the set of edges that \emph{belong to $K_v$ as a subgraph} is later denoted $E_o(K_v)$).

\begin{corollary}\label{cor:cluster_repulsion}
Let $G=(V,E)$ be a connected, locally finite graph and suppose that $p_c \leq p_1 <1$, $\phi:(0,\infty)\to (0,\infty)$, and $c>0$ are such that 
\[
\sum_{e \in \partial_E^\rightarrow S} \P_{p_1}\left(e^+ \leftrightarrow \infty \text{ \emph{off} $S$}\right) \geq c \cdot \phi(|E(S)|)
\]
for every $S \subseteq V$ finite. Then 
\[
\P_{p_2}(|E(K_v)|=n) \leq \frac{2n}{c\phi(n)}\left(\frac{1-p_2}{1-p_1}\right)^{c\phi(n)/2}
\]
for every $p_1 <p_2 \leq 1$ and $n\geq 1$.
\end{corollary}

The proof of this corollary will use
the fact that if $X$ is a random variable taking values in $[0,M]$ for some $0 < M<\infty$ then
\begin{equation}
\label{eq:Markov}
\P\left(X > \theta \E X\right) = 1-\P\left(M-X \geq M- \theta \E X\right) \geq 1-\frac{M-\E X}{M- \theta\E X} = \frac{(1-\theta)\E X}{M- \theta\E X} \geq  (1-\theta)\frac{\E X}{M}
\end{equation}
for every $0<\theta<1$, 
where we applied Markov's inequality to $M-X$ in the central inequality.

\begin{proof}[Proof of \cref{cor:cluster_repulsion} given \cref{prop:cluster_repulsionI}]
Fix $p_1 < p_2 \leq 1$, $\phi$, and $c$ as in the statement of the corollary. We couple $G_{p_1}$ and $G_{p_2}$ in the standard monotone way, let $K_1=K_{v,p_1}$, let $K_2=K_{v,p_2}$, and let $Z=\tau(K_{v,p_2},K_{\infty,p_1})$. Conditional on $K_2$, the restriction of $G_{p_1}$ to the subgraph induced by $K_2^c=V\setminus K_2$ is distributed as Bernoulli-$p_1$ bond percolation on this subgraph. As such, we have by assumption that
\[
\E\left[ Z \mid K_2 \right]
\geq \min\left\{\E_{p_1}\left[|\{e\in \partial_E^\rightarrow S: e^+ \xleftrightarrow{p_1} \infty \text{ off $S$}\}|\right]: S \subseteq V,\, |E(S)|=|E(K_2)|\right\}
\geq c \phi(|E(K_2)|)
\]
almost surely when $K_2$ is finite. Applying \eqref{eq:Markov} to the conditional distribution of $Z$ given $|E(K_2)|=n$ with $M=n$ and $\theta=1/2$, it follows that
\begin{equation}
\label{eq:Z_Markov}
\P\left(Z \geq \frac{c}{2}\phi(n) \mid |E(K_2)|=n \right) \geq \P\left(Z \geq \frac{1}{2}\E\left[Z \mid |E(K_2)|=n\right] \mid |E(K_2)|=n \right) \geq \frac{c\phi(n)}{2n}
\end{equation}
for every $n\geq 1$.
On the other hand, we also have by \cref{prop:cluster_repulsionI} that
\begin{align*}\P\left(Z \geq \frac{c}{2}\phi(n) \mid |E(K_2)|=n \right)
&= \P\left(|K_2|<\infty, Z \geq \frac{c}{2}\phi(n) \mid |E(K_2)|=n \right)
\\ &\leq 
\P\left(|K_2|<\infty, Z \geq \frac{c}{2}\phi(n) \right) \P(|E(K_2)|=n)^{-1}\\ &\leq \left(\frac{1-p_2}{1-p_1}\right)^{c\phi(n)/2} \P(|E(K_2)|=n)^{-1},
 \end{align*}
 for every $n\geq 1$, which yields the claimed inequality when compared with \eqref{eq:Z_Markov}.
\end{proof}

We now begin the proof of \cref{prop:cluster_repulsionI}. We begin by making note of the following fact.

\begin{lemma}\label{lem:exploration}
Let $G=(V,E)$ be a connected, locally finite graph, let $0<p_1<p_2<1$, and consider the standard monotone coupling of $G_{p_1}$ and $G_{p_2}$. Let $K_{\infty,p_1}$ be the set of vertices connected to infinity in $G_{p_1}$. Conditional on $K_{\infty,p_1}$, the edges of $G_{p_2}$ that have both endpoints in the complement $K_{\infty,p_1}^c=V\setminus K_{\infty,p_1}$ are distributed as Bernoulli-${p_2}$ bond percolation on the subgraph of $G$ induced by $K_{\infty,p_1}^c$.
\end{lemma}

\begin{proof}[Proof of \cref{lem:exploration}]
Let $(V_n)_{n\geq 0}$ be an exhaustion of $V$ by finite sets, and for each $n\geq 1$ let $K_{n,p_1}$ denote the set of vertices connected to $V_n^c=V\setminus V_n$ in $G_{p_1}$. It is clear that if we condition on $K_{n,p_1}$, the edges of $G_{p_2}$ that have both endpoints in the complement $K_{n,p_1}^c=V\setminus K_{n,p_1}$ are distributed as Bernoulli-$p_2$ bond percolation on the subgraph of $G$ induced by $K_{n,p_1}^c$. The claim follows by taking the limit as $n\to\infty$ and noting that $K_{n,p_1} \to K_{\infty,p_1}$ as $n\to\infty$ since $G$ is connected and locally finite.
\end{proof}

\begin{proof}[Proof of \cref{prop:cluster_repulsionI}]
Let $K'_{v,p_2}$ be the cluster of $v$ in the subgraph of $G_{p_2}$ induced by $K_{\infty,p_1}^c$. If $K_{v,p_2}$ is finite and has $m$ edges in its boundary touching $K_{\infty,p_1}$,  the cluster $K'_{v,p_2}$ must be finite and have $m$ edges in its boundary touching $K_{\infty,p_1}$, all of which are closed in $G_{p_2}$. The conditional probability that all these edges are closed given $K_{v,p_2}'$ and $K_{\infty,p_1}$ is $(1-p_2)^m/(1-p_1)^m$, so that
\begin{align*}
&\P\left(|K_{v,p_2}|<\infty \text{ and } \tau\bigl(K_{v,p_2},K_{\infty,p_1}\bigr) \geq n\right)\\ &\hspace{4.5cm}= \sum_{m\geq n} \P\left(|K_{v,p_2}'|<\infty \text{ and } \tau\bigl(K_{v,p_2},K_{\infty,p_1}\bigr) = m\right)\\
&\hspace{4.5cm}= \sum_{m\geq n} \left(\frac{1-p_2}{1-p_1}\right)^m\P\left(|K_{v,p_2}'|<\infty \text{ and } \tau\bigl(K_{v,p_2}',K_{\infty,p_1}\bigr) = m\right) \\
&\hspace{4.5cm} \leq \left(\frac{1-p_2}{1-p_1}\right)^n\P\left(|K_{v,p_2}'|<\infty \text{ and } \tau\bigl(K_{v,p_2}',K_{\infty,p_1}\bigr) \geq n\right) \leq \left(\frac{1-p_2}{1-p_1}\right)^n
\end{align*}
for every $n\geq 1$ as claimed.
\end{proof}



\subsection{Cluster repulsion via martingales}
\label{subsec:martingale}

In this section we prove another cluster repulsion inequality that replaces the need to sprinkle with the condition that $|K_v|$ is much smaller than $\tau(K_v,K_\infty)^2$. This inequality is not needed for the proof of the main theorem but is included since it is of independent interest. The proof is inspired by Aizenman, Kesten, and Newman's proof of the uniqueness of the infinite cluster in $\Z^d$ \cite{MR901151} (see also \cite{MR3395466}), and more specifically on the martingale arguments we used to prove generalisations of this inequality in \cite{1808.08940}.

\begin{prop}[Cluster repulsion via martingales]
\label{prop:cluster_repulsionII}  Let $G=(V,E)$ be a connected, locally finite graph, and let $0<p<1$. The inequality
\[
\P\left(|E(K_v)| \leq m \text{ and } \tau(K_v,K_\infty) \geq n\right) \leq 2\exp\left[-\frac{p^2 n^2}{8 m} \right]
\]
holds for every $v\in V$ and $n,m\geq 1$.
\end{prop}

Since this inequality is not needed for the proof of our main theorems we will be a little sketchy in the details.

\begin{proof}[Sketch of proof]
We will define two bounded-increment martingales, with respect to different filtrations, such that at least one of the two martingales is atypically large if $\tau(K_v,K_\infty) \gg |E(K_v)|^{1/2}$. 


Let us first recall the usual method for producing a martingale by exploring a percolation cluster, referring the reader to  \cite[Theorem 1.6]{1808.08940} for further details. 
If $H$ is a locally finite graph and $v$ is a vertex of $H$, we can algorithmically explore the cluster of $v$ in $H_p$ in an edge-by-edge fashion by first fixing an enumeration of the edges of $H$ and, at each step, querying the status of the edge that is minimal with respect to this enumeration among those edges that are adjacent to the revealed part of the cluster and have not yet been queried. Moreover, if $E_i$ denotes the $i$th edge queried by this procedure and $T$ denotes the first time that we have queried every edge touching $K_v$, which is finite if and only if $K_v$ is finite, then
\[
Z_n = \sum_{i=1}^{n \wedge T} \left[(1-p)\mathbbm{1}(E_i \text{ open})-p\mathbbm{1}(E_i \text{ closed})\right]
\]
is a martingale with respect to its natural filtration and satisfies
\[
Z_T = (1-p)\#\{\text{open edges of $K_v$}\}-p\#\{\text{closed edges touching $K_v$}\}.
\]
In particular, the final value of $Z$ does not depend on the specific details of the exploration procedure such as the choice of enumeration.


We now consider producing such a martingale in two different ways. Fix a graph $G$, $0<p<1$, and a vertex $v$ of $G$. The first martingale, $Z$, is just the usual exploration martingale computed when exploring the cluster of $v$ in $G_p$ as above, which is a martingale with respect to its natural filtration $\mathcal{F}_n=\sigma \langle Z_1,\ldots,Z_n \rangle$. For the second martingale, $\tilde Z$, we first explore every infinite cluster of $G_p$ and then explore the cluster of $v$ in the subgraph induced by $G \cap K_\infty^c$, letting $\tilde E_i$ be the $i$th edge touching the cluster of $v$ in this graph that is queried in the second step of this algorithm, and letting
\[
\tilde Z_n = \sum_{i=1}^{n \wedge \tilde T} \left[(1-p)\mathbbm{1}(\tilde E_i \text{ open})-p\mathbbm{1}(\tilde E_i \text{ closed})\right]
\]
where $\tilde T$ denotes the first time that all edges touching the cluster of $v$ have been queried. In particular, if $v$ belongs to $K_\infty$ then $\tilde T = \tilde Z_{\tilde T} = 0$. Similarly to above, one readily verifies that $\tilde Z$ is a martingale with respect to the filtration $\tilde{\mathcal{F}}_n = \sigma \langle K_\infty, \tilde Z_1,\ldots,\tilde Z_n\rangle$ and satisfies
\begin{multline*}
\tilde Z_{\tilde T} =\\ \left[(1-p)\#\{\text{open edges of $K_v$}\}-p\#\{\text{closed edges touching $K_v$}\}+p\cdot \tau(K_v,K_\infty)\right] \mathbbm{1}(v \nleftrightarrow \infty).
\end{multline*}
To see why this holds, note that the edges counted by $\tau(K_v,K_\infty)$ do not need to be queried in the second step of the algorithm since they are already queried in the first step (which does not contribute to the martingale by definition). Thus, we have that
\[
\mathbbm{1}(v\nleftrightarrow \infty)\left(\tilde Z_{\tilde T}-Z_T\right) =   p\cdot \tau(K_v,K_\infty) 
\]
and since $\tilde T \leq T = |E(K_v)|$ it follows that
\begin{multline*}
\P\left(|E(K_v)| \leq m \text{ and } \tau(K_v,K_\infty) \geq n\right) \\\leq \P\left(\max_{0\leq n \leq m} (-Z_n) \geq \frac{pn}{2}\right)
+
\P\left(\max_{0\leq n \leq m} \tilde Z_n \geq \frac{pn}{2}\right) \leq 2 \exp\left[ - \frac{p^2n^2}{8 m}\right]
\end{multline*}
as claimed; in the final inequality we have used the maximal version of the Azuma-Hoeffding inequality \cite[Section 2]{McDiarmid1998} and the fact that $Z$ and $\tilde Z$ are martingales with increments bounded by $1$.
\end{proof}

\section{Isoperimetry from tail estimates}
\label{subsec:isoperimetry}

In this section we prove the following proposition that converts upper bounds on the tail of the volume of finite clusters into isoperimetric estimates in the infinite cluster. The proof is a simple elaboration of an argument due to Pete \cite[Theorem A.1]{chen2004anchored} (see also \cite[Theorem 4.1]{Pete08}). We define $\partial_p W$ to be the edge boundary of $W$ inside the random graph $G_p$ and recall that $E(W)$ denotes the set of edges that touch $W$.

\begin{prop}
\label{prop:union_bound_isoperimetry}
Let $G$ be a bounded degree graph, let $v$ be a vertex of $G$, let $p_c(G) \leq p < 1$, and let $\phi:(0,\infty)\to (0,\infty)$ be an increasing function with $\phi(t)\leq t$ for every $t> 0$. If there exists $0<c\leq 2$ such that
$\P_p(|E(K_v)| =n) \leq \exp\left[ -c \phi(n)\right]$
for every $n\geq 1$ then for each $\eps>0$ there exist a constant $C=C(\eps,c,\phi)$ such that if $p \leq 1-\eps$ then
\begin{multline*}
\P_p\left(\text{$\exists W \subseteq K_v$ connected with $v\in W$, $|E(W)|= n$, and $|\partial_p W| \leq \frac{c}{4}\cdot\frac{\phi(|W|)}{\log \frac{2|W|}{\phi(|W|)}}$}\right)\\ \leq C\exp\left[ -\frac{c}{2} \phi(n)\right]
\end{multline*}
for every $n\geq 1$. In particular, if $\phi(t)/\log t \to \infty$ as $t\to\infty$ and we define $\psi:(0,\infty)\to(0,\infty)$ by $\psi(t)=\phi(t)/\log (2t/\phi(t))$ then $K_v$ satisfies a $\psi$-anchored isoperimetric inequality almost surely on the event that it is infinite.
\end{prop}

\begin{proof}[Proof of \cref{prop:union_bound_isoperimetry}]
Fix $\eps>0$ and $p_c<p\leq 1-\eps$ and let $\delta = c/4 \leq 1/2$. For each $n\geq 1$ let $\sA_{n}$ be the event whose probability is being estimated and  let $\sH_v^n$ be the set of connected subgraphs of $G$ containing $v$ that touch exactly $n$ edges.  We have by a union bound that
\begin{equation}
\label{eq:ExpansionUnionBound}
\P_p(\sA_{n}) \leq \sum_{m=1}^{\lfloor \delta\psi(n)\rfloor} \sum_{H \in \sH_v^n} \bP_p\left(H \subseteq K_v,\, |\partial_p H| =m\right).
\end{equation}
For each $H\in \sH_v$, let $E(H)$ denote the set of edges that touch the vertex set of $H$, let $E_o(H)$ be the set of edges belonging to $H$, and let $\partial H = E(H) \setminus E_o(H)$ be the set of edges that touch but do not belong to $H$.
Given $H \in \sH_v$ and $S \subseteq \partial H$, on the event that
$H \subseteq K_v$ and $\partial_p H  = S$ each edge $e \in E(H)$ is open if and only if $e \in E_o(H) \cup S$, so that
\begin{align}
\bP_p( H \subseteq K_v \text{ and } \partial_p H  = S) \leq p^{|E_o(H)|+|S|}(1-p)^{|\partial H|-|S|} 
=  \left( \frac{p}{1-p}\right)^{|S|} \bP_p(K_v=H).
\end{align}
Using another union bound and summing over the possible choices of $S$ with $S \subseteq E(H)$ and $|S|=m$, we deduce that
\[
\bP_p( H \subseteq K_v \text{ and } |\partial_p H | =m) \leq 
 \binom{E(H)}{m} \left(\frac{p}{1-p}\right)^m \bP_p(K_v=H) 
\]
and hence that
\begin{align}
\bP_p( \sA_{n}) &\leq  \sum_{m=1}^{\lfloor \delta \psi(n) \rfloor} \sum_{H \in \sH_v^n} \bP_p\bigl( H \subseteq  K_v,\, |\partial_p H| =m\bigr) \leq    \bP_p(E_v=n) \sum_{m=1}^{\lfloor \delta \psi(n)\rfloor}  \binom{n }{m} \left(\frac{p}{1-p}\right)^m. \label{eq:union_bound1}
\end{align}
Since $\binom{n}{m}$ is increasing in $m$ for $m\leq n/2$, it follows from the elementary bound $\binom{n}{m} \leq e^m(n/m)^m$ that there exist constants $C_1=C_1(c,\phi,\eps)$ and $C_2=C_2(c,\phi,\eps)$ such that
\begin{align}
\sum_{m=1}^{\lfloor \delta \psi(n)\rfloor}  \binom{n }{m} \left(\frac{p}{1-p}\right)^m &\leq 
\delta \psi(n) \exp\left[ \delta \psi(n) \left(\log \frac{2n}{\psi(n)}+\log \frac{1}{2\delta} + \log \max\left\{1,\frac{p}{1-p}\right\}+1\right) \right]\nonumber\\
&\leq   \exp\left[ \frac{\delta \phi(n)}{\log \frac{2n}{\phi(n)}} \left(\log \frac{2n}{ \phi(n)} +\log \log \frac{2n}{\phi(n)} +C_1\right) + \log \frac{\delta\phi(n)}{\log \frac{2n}{\phi(n)}}\right]\nonumber\\
&\leq C_2\exp\left[2 \delta \phi(n)\right] = C_2\exp\left[\frac{c}{2} \phi(n)\right],\label{eq:union_bound2}
\end{align}
where in the final inequality we were able to absorb all lower order terms into the constant by doubling the coefficient on the leading order term in the exponent.
It follows from \eqref{eq:union_bound1} and \eqref{eq:union_bound2} that
\[
\P_p(\sA_{n}) \leq C_2\P_p(E_v=n) e^{\frac{c}{2} \phi(n)} \leq C_2 e^{-c \phi(n)+\frac{c}{2} \phi(n)} = C_2 e^{-\frac{c}{2}\phi(n)}
\]
as claimed.
\end{proof}

\section{Proof of the main theorem}
\label{subsec:mainproof}

\begin{proof}[Proof of \cref{thm:main_equivalence}]
Fix $G$, $p_0$, and $d$ as in the statement of the theorem.
We will prove the implications (i) $\Rightarrow$ (ii), (ii) $\Rightarrow$ (i), (ii) $\Rightarrow$ (iii), (iii) $\Rightarrow$ (iv), and (iv) $\Rightarrow$ (ii). Only the proof of (iv) $\Rightarrow$ (ii) will require transitivity. 

\medskip
\noindent 
\textbf{(i) $\Rightarrow$ (ii):} Given an increasing event $A$ and an integer $r\geq 1$, let $I_r(A)$ denote the event that $A$ holds in $G_p \setminus F$ for every set of edges $F$ of size at most $r$. An inequality due to Aizenman, Chayes, Chayes, Fr\"ohlich, and Russo \cite{aizenman1983sharp} (see also \cite[Theorem 2.45]{grimmett2010percolation}) states that
\begin{equation}
\label{eq:stability}
\P_{p_2}(I_r(A)) \geq 1-\left(\frac{p_2}{p_2-p_1}\right)^r\bigl[1-\P_{p_1}(A)\bigr]
\end{equation}
for every $0<p_1<p_2<1$, every increasing event $A$, and every $r\geq 1$. 
Note that if $S$ is a finite set of vertices in $G$ then $I_r(\{S\leftrightarrow \infty\})$ is equal by Menger's theorem (or max-flow min-cut) to the event that there are at least $r+1$ edge-disjoint paths connecting $S$ to $\infty$.
Fix $p_0 < p_1 < p_2 <1$ and $d'<d$ and let $c_1>0$ be such that $\P_{p_1}(S \nleftrightarrow \infty) \leq \exp\left[-c_1 |S|^{(d'-1)/d'} \right]$ for every finite set of vertices $S$. It follows from \eqref{eq:stability} that there exist positive constants $c_2=c_2(p_1,p_2,c_1)$ and $c_3=1-e^{-c_2/2}$ such that if $r= \lfloor c_2 |S|^{(d'-1)/d'}\rfloor$ then
\[\P_{p_2}(I_r(\{S \leftrightarrow \infty\})) \geq 1-\left(\frac{p_2}{p_2-p_1}\right)^r \exp\left[-c_1 |S|^{(d'-1)/d'}\right] \geq 1- \exp\left[-\frac{c_1}{2} |S|^{(d'-1)/d'}\right] \geq c_3. \]
Since $\sum_{e \in \partial_E^\rightarrow S} \mathbbm{1}(e^+ \leftrightarrow \infty$ off $S)$ is lower bounded by the maximum size of a collection of edge-disjoint paths connecting $S$ to infinity, we deduce that there exists a positive constant $c_4$ such that
\[
\sum_{e \in \partial_E^\rightarrow S} \P_{p_2}(e^+ \leftrightarrow \infty \text{ off }S) \geq (r+1) \P_{p_2}(I_r(\{S \leftrightarrow \infty\})) \geq c_3 (1+\lfloor c_2 |S|^{(d'-1)/d'}\rfloor) \geq c_4 |S|^{(d'-1)/d'}
\]
as required.

\medskip
\noindent
\textbf{(ii) $\Rightarrow$ (i):} For each finite set $S \subseteq V$ and $0<p<1$ let $\Psi_p(S) = \sum_{e\in \partial_E^\rightarrow S} \P_p(e^+ \leftrightarrow \infty$ off $S)$. Russo's formula states that if $A$ is an increasing event depending on at most finitely many edges then 
\[
\frac{d}{dp} \P_p(A) = \frac{1}{1-p}\sum_{e\in E}\P_p(e \text{ is closed pivotal for $A$}),
\]
where an edge $e$ is said to be closed pivotal for $A$ if $G_p \notin A$ but $G_p \cup \{e\} \in A$. Without the assumption that $A$ depends on at most finitely many edges, we still have the inequality
\begin{equation}
\label{eq:Russo_Dini}
\lrDini \P_p(A) \geq \frac{1}{1-p}\sum_{e\in E}\P_p(e \text{ is closed pivotal for $A$}),
\end{equation}
where $\lrDini \P_p(A) = \liminf_{\eps \downarrow 0} \frac{1}{\eps}(\P_{p+\eps}(A)-\P_p(A))$ is the \emph{lower-right Dini derivative} of $\P_p(A)$. Let $S$ be a finite set of vertices and let $K_S$ be the union of all clusters intersecting $S$. When $K_S$ is finite, an edge $e$ is a closed pivotal for the event $\{S \leftrightarrow \infty\}$ if it has one endpoint in $K_S$ and the other in an infinite cluster. Since conditional on $K_S$ the edges of $G_p$ that do not touch $K_S$ are distributed as Bernoulli-$p$ percolation on the subgraph of $G$ induced by $K_S^c$, it follows that
\begin{align}
\lrDini \P_p(S \leftrightarrow \infty) &\geq \frac{1}{1-p} \E_p\left[\Psi_p(K_S) \mid S \nleftrightarrow \infty\right] \P_p(S \nleftrightarrow \infty)\nonumber\\
&\geq \frac{1}{1-p} \min\{\Psi_p(W) : W \supseteq S \text{ finite}\} \P_p(S \nleftrightarrow \infty). \label{eq:RussoS}
\end{align}
Thus, if $p_0<p_1<1$, $d'<d$, and $c=c(p_1,d')>0$ are such that $\Psi_{p_1}(S) \geq c |S|^{(d'-1)/d'}$ for every finite set of vertices $S$, it follows from \eqref{eq:RussoS} and the fact that $\Psi_p$ is increasing in $p$ that
\[
\lrDini \log \frac{1}{\P_p(S \nleftrightarrow \infty)} \geq \frac{c}{1-p}|S|^{(d'-1)/d'}
\]
for every $p_1\leq p<1$. Integrating this inequality yields that
\[
\P_{p_2}(S \nleftrightarrow \infty) \leq \exp\left[-\frac{c(p_2-p_1)}{1-p_1}|S|^{(d'-1)/d'}\right]
\]
for every $p_0 < p_1 <p_2<1$ and every finite set $S$ as required.

\medskip
\noindent 
\textbf{(ii) $\Rightarrow$ (iii):} This follows immediately from \cref{cor:cluster_repulsion} applied with $\phi(t)=t^{(d'-1)/d'}$.

\medskip
\noindent 
\textbf{(iii) $\Rightarrow$ (iv):} This follows immediately from \cref{prop:union_bound_isoperimetry} applied with $\phi(t)=t^{(d'-1)/d'}$.

\medskip
\noindent 
\textbf{(iv) $\Rightarrow$ (ii):} Given a connected graph $H$, a vertex $v$ of $H$, and $d' \geq 1$, let
\[
\Phi^*_{d'}(H,v) = \inf \left\{\frac{|\partial_E W|}{|W|^{(d'-1)/d'}} : W \text{ a finite, connected set of vertices containing $v$}\right\},
\]
so that if $H$ has bounded degrees then it satisfies an anchored $d'$-dimensional isoperimetric inequality if and only if $\Phi^*_{d'}(H,v)>0$ for every vertex $v$ of $H$. Let $p_0<p_1$ and $d'<d$. Since $G$ is transitive and the infinite clusters of $G_{p_1}$ satisfy an anchored $d'$-dimensional isoperimetric inequality almost surely, there exists $\eps=\eps(p_1,d')>0$ such that $\P_{p_1}(\Phi^*_{d'}(K_v,v) \geq \eps) \geq \eps$ for every vertex $v$ of $G$. Fix such an $\eps>0$ and let $A$ be the random set of vertices $v$ whose cluster in $G_{p_1}$ is infinite and satisfies $\Phi^*_{d'}(K_v,v) \geq \eps$. Thus, if $S$ is a finite set of vertices we have by linearity of expectation that $\E_{p_1}|S \cap A| \geq \eps |S|$ and hence by Markov's inequality as in \eqref{eq:Markov} that 
\[\P_{p_1}\left(|S \cap A| \geq \frac{\eps}{2}|S|\right) \geq \frac{\eps}{2}.\]
Thus, to complete the proof it suffices to prove that
\[
\sum_{e\in \partial_E^\rightarrow S} \mathbbm{1}\bigl(e^+ \leftrightarrow \infty \text{ off } S\bigr) \geq \eps|A \cap S|^{(d'-1)/d'}
\]
for every finite set of vertices $S$. 
 Fix one such set $S$ and let the \textbf{hull} $\Gamma(S \cap K_\infty) \supseteq S \cap K_\infty$ be the set of vertices $v$ such that $K_v$ is infinite but any open path connecting $v$ to infinity must pass through $S$. This definition ensures that $\{e \in \partial_E^\rightarrow S : e^+ \leftrightarrow \infty$ off  $S\}$ is equal to the oriented edge boundary $\Gamma(S \cap K_\infty)$ in $G_p$, which we denote by $\partial_p^\rightarrow \Gamma(S \cap K_\infty)$. Letting the connected components of $\Gamma(S \cap K_\infty)$ be enumerated $C_1,\ldots,C_m$, we have by definition of $A$ that
 \[
|\partial_p^\rightarrow \Gamma(S \cap K_\infty)| = \sum_{i=1}^m |\partial_p^\rightarrow C_i| \geq \sum_{i=1}^m \eps |C_i|^{(d'-1)/d'} \mathbbm{1}(C_i \cap A \neq \emptyset) \geq \sum_{i=1}^m \eps |C_i \cap A|^{(d'-1)/d'} \geq \eps |A \cap S|^{(d'-1)/d'}
 \]
 as required.
\end{proof}

\begin{remark}
Note that the proof of the implication (iv) $\Rightarrow$ (ii) works for any bounded degree graph in which the assumption that infinite $p_1$-clusters satisfy an anchored $d'$-dimensional inequality implies that there exists $\eps>0$ such that $\P_{p_1}(\Phi^*_{d'}(K_v,v) \geq \eps) \geq \eps$ for every vertex $v$ of $G$. It is easily seen that this also holds for quasi-transitive graphs.
\end{remark}

Let us end the paper with the following question that arose during this work. 

\begin{question}
\label{question}
Does the infinite cluster of supercritical percolation on $\Z^d$ admit a positive density subgraph with uniform isoperimetric dimension $d$? Does the analogous statement hold for other transitive graphs?
\end{question}

Note that the analogous question in the nonamenable setting admits a positive answer due to Benjamini, Lyons, and Schramm \cite[Theorem 1.1]{BLS99} (see also \cite{v00}). Unfortunately however it is not true in general that a unimodular random rooted graph with anchored isoperimetric dimension $d$ always admits a positive density subgraph with uniform isoperimetric dimension $d$. Indeed, by stretching the edges of $\Z^d$ that lie on the boundaries of large blocks in a uniform hierarchical decomposition of $\Z^d$ one can obtain a unimodular random graph that has anchored isoperimetric dimension $d$ but does not admit any infinite subgraphs satisfying any non-trivial uniform isoperimetric inequality. Note also that Grimmett, Holroyd, and Kozma \cite{grimmett2014percolation} have shown that the infinite cluster in supercritical percolation on $\Z^d$ never contains a quasi-isometric copy of $\Z^d$, ruling out one strategy to answer \cref{question}.


\subsection*{Acknowledgements} We thank Russ Lyons his careful reading and helpful comments on an earlier version of this manuscript and thank Philip Easo for catching several typos.

\addcontentsline{toc}{section}{References}

\footnotesize{

 \setstretch{1}
  \bibliographystyle{abbrv}
  \bibliography{unimodularthesis.bib}

  }
\end{document}